\documentclass[12pt]{amsart}
\usepackage{amssymb}
\usepackage{amsfonts}
\usepackage{latexsym}
\usepackage{amsmath}

\font\Bbb=msbm10 at 12pt
\newcommand{\R}{\mbox{\Bbb R}}
\newcommand{\N}{\mbox{\Bbb N}}

\theoremstyle{plain}
\newtheorem{theorem}{Theorem}
\newtheorem{cor}[theorem]{Corollary}
\newtheorem{lemma}[theorem]{Lemma}

\theoremstyle{definition}
\newtheorem{definition}{Definition}

\theoremstyle{remark}
\newtheorem{remark}{Remark}

\newcommand{\tx}[1]{\quad\mbox{ #1 }\quad}

\numberwithin{equation}{section}
\numberwithin{theorem}{section}
\numberwithin{definition}{section}
\numberwithin{remark}{section}

\small\normalsize
\begin{document}

\title[fractional taylor's formula]{taylor's formula and integral inequalities for conformable fractional derivatives}

\author[D. R. Anderson]{Douglas~R.~Anderson}
\address{Concordia College, Department of Mathematics, Moorhead, MN 56562 USA}
\email{andersod@cord.edu}


\begin{abstract} 
We derive Taylor's Formula for conformable fractional derivatives. This is then employed to extend some recent and classical integral inequalities to the conformable fractional calculus, including the inequalities of Steffensen, $\breve{\text{C}}$eby$\breve{\text{s}}$ev, Hermite-Hadamard, Ostrowski, and Gr\"{u}ss.
\end{abstract}

\maketitle\thispagestyle{empty}


\section{Taylor Theorem}
We use the conformable $\alpha$-fractional derivative, recently introduced in \cite{hammad,khalil}, which for $\alpha\in(0,1]$ is given by
\begin{equation}\label{derivdef}
 D_{*}f(t):=\lim_{\varepsilon\rightarrow 0}\frac{f(t+\varepsilon t^{1-*})-f(t)}{\varepsilon}, \quad D_{*}f(0)=\lim_{t\rightarrow 0^+}D_{*}f(t).  
\end{equation}
Note that if $f$ is differentiable, then
\begin{equation}\label{fracshort}
 D_{\alpha}f(t) = t^{1-\alpha} f'(t), 
\end{equation}
where $f'(t)=\lim_{\varepsilon\rightarrow 0}[f(t+\varepsilon)-f(t)]/\varepsilon$.

We will consider Taylor's Theorem in the context of iterated fractional differential equations. In this setting, the theorem will be proven using the variation of constants formula, where we use an approach similar to that used for integer-order derivatives found in \cite{kp}.  With this in mind, we begin this note with a general higher-order equation. For $n\in\N_0$ and continuous functions $p_i:[0,\infty)\to\R$, $1\le i\le n$, we consider the higher-order linear $\alpha$-fractional differential equation\index{higher order linear equation}
\begin{equation}\label{13}
  Ly=0,\tx{where}Ly = D^n_{\alpha}y+\sum_{i=1}^n p_i D^{n-i}_{\alpha}y,
\end{equation}
where $D^n_{\alpha}y=D^{n-1}_{\alpha}(D_\alpha y)$. 
A function $y:[0,\infty)\to\R$ is a solution of equation \eqref{13} on $[0,\infty)$ provided $y$ is $n$ times $\alpha$-fractional differentiable
on $[0,\infty)$ and satisfies $Ly(t)=0$ for all $t\in[0,\infty)$. It follows that $D^n_{\alpha}y$ is a continuous function on $[0,\infty)$.

Now let $f:[0,\infty)\to\R$ be continuous and consider the non-homogeneous equation
\begin{equation}\label{78}
  D^n_{\alpha}(t)+\sum_{i=1}^np_i(t)D^{n-i}_{\alpha}y(t)=f(t).
\end{equation}

\begin{definition} 
	\index{Cauchy function!higher order equation}
We define the Cauchy function $y:[0,\infty)\times[0,\infty)\to\R$ for the linear fractional equation \eqref{13} to be, for each fixed $s\in[0,\infty)$, the solution of the initial value problem 
$$ Ly=0,\quad D^i_{\alpha}y(s,s)=0,\quad 0\le i\le n-2,\quad  D^{n-1}_{\alpha}y(s,s)=1. $$
\end{definition}

\begin{remark}
Note that 
$$ y(t,s):= \frac{1}{(n-1)!}\left(\frac{t^\alpha-s^\alpha}{\alpha}\right)^{n-1} $$ 
is the Cauchy function for $D^n_{\alpha}=0$, which can be easily verified using \eqref{fracshort}.
\end{remark}


\begin{theorem}[Variation of Constants]\label{t722}
	\index{variation of parameters!higher order}
Let $\alpha\in(0,1]$ and $s,t\in[0,\infty)$.  If $f$ is continuous, then the solution of the initial value problem
$$ Ly=f(t),\quad D^i_{\alpha}y(s)=0,\quad 0\le i\le n-1 $$  
is given by
$$ y(t) = \int_{s}^t y(t,\tau)f(\tau)\tau^{\alpha-1} d\tau, $$
where $y(t,\tau)$ is the Cauchy function for \eqref{13}.
\end{theorem}

\begin{proof} 
With $y$ defined as above and by the properties of the Cauchy function we have
$$ D^i_{\alpha}y(t)
   =\int_{s}^t D^i_{\alpha}y(t,\tau)f(\tau)\tau^{\alpha-1}d\tau + D^{i-1}_{\alpha}y(t,t)f(t)
   =\int_{s}^t D^i_{\alpha}y(t,\tau)f(\tau)\tau^{\alpha-1}d\tau $$
for $0\le i\le n-1$, and
\begin{eqnarray*} 
  D^n_{\alpha}y(t)
   &=& \int_{s}^t D^n_{\alpha}y(t,\tau)f(\tau)\tau^{\alpha-1}d\tau + D^{n-1}_{\alpha}y(t,t)f(t) \\
   &=& \int_{s}^t D^n_{\alpha}y(t,\tau)f(\tau)\tau^{\alpha-1}d\tau+f(t). 
\end{eqnarray*}
It follows from these equations that 
$$ D^{i}_{\alpha}y(s)=0, \quad 0\le i\le n-1 $$
and
$$ Ly(t)=\int_{s}^t Ly(t,\tau)f(\tau)\tau^{\alpha-1}d\tau+f(t)=f(t), $$
and the proof is complete.
\end{proof}


\begin{theorem}[Taylor Formula]\label{t144}
Let $\alpha\in(0,1]$ and $n\in\N$. Suppose $f$ is $(n+1)$ times $\alpha$-fractional differentiable on $[0,\infty)$, and $s,t\in[0,\infty)$. Then we have
$$  f(t) = \sum_{k=0}^n  \frac{1}{k!}\left(\frac{t^\alpha-s^\alpha}{\alpha}\right)^{k} D^{k}_{\alpha}f(s)
         +\frac{1}{n!} \int_{s}^t \left(\frac{t^\alpha-\tau^\alpha}{\alpha}\right)^n D^{n+1}_{\alpha}f(\tau)\tau^{\alpha-1}d\tau $$
\end{theorem}

\begin{proof}
Let $g(t):=D^{n+1}_{\alpha}f(t)$.  Then $f$ solves the initial value problem
$$ D^{n+1}_{\alpha}x = g, \quad D^{k}_{\alpha}x(s)=D^{k}_{\alpha}f(s), \quad 0\le k\le n. $$
Note that the Cauchy function for $D^{n+1}_{\alpha}y=0$ is 
\[ y(t,s) = \frac{1}{n!}\left(\frac{t^\alpha-s^\alpha}{\alpha}\right)^{n}. \]
By the variation of constants formula, 
$$ f(t)=u(t) + \frac{1}{n!}\int_{s}^t \left(\frac{t^\alpha-\tau^\alpha}{\alpha}\right)^{n}g(\tau)\tau^{\alpha-1}d\tau,  $$
where $u$ solves the initial value problem 
\begin{equation}\label{uivp}
   D^{n+1}_{\alpha}u=0, \quad D^{m}_{\alpha}u(s) = D^{m}_{\alpha}f(s), \quad 0\le m \le n. 
\end{equation}
To validate the claim that $u(t)=\sum_{k=0}^n  \frac{1}{k!}\left(\frac{t^\alpha-s^\alpha}{\alpha}\right)^{k} D^{k}_{\alpha}f(s)$, set
$$ w(t):=\sum_{k=0}^n  \frac{1}{k!}\left(\frac{t^\alpha-s^\alpha}{\alpha}\right)^{k} D^{k}_{\alpha}f(s). $$
Then $D^{n+1}_{\alpha}w=0$, and we have that
$$ D^{m}_{\alpha}w(t) = \sum_{k=m}^n \frac{1}{(k-m)!}\left(\frac{t^\alpha-s^\alpha}{\alpha}\right)^{k-m} D^{k}_{\alpha}f(s). $$
It follows that 
$$ D^{m}_{\alpha}w(s) = \sum_{k=m}^n \frac{1}{(k-m)!}\left(\frac{s^\alpha-s^\alpha}{\alpha}\right)^{k-m} D^{k}_{\alpha}f(s) =  D^{m}_{\alpha}f(s) $$
for $0\le m\le n$.  We consequently have that $w$ also solves \eqref{uivp}, and thus $u\equiv w$ 
by uniqueness.
\end{proof}


\begin{cor}
Let $\alpha\in(0,1]$ and $s,r\in[0,\infty)$ be fixed. For any $t\in[0,\infty)$ and any positive integer $n$, 
$$ \frac{1}{n!}\left(\frac{t^\alpha-r^\alpha}{\alpha}\right)^{n} = \sum_{k=0}^n \frac{1}{k!(n-k)!}\left(\frac{t^\alpha-s^\alpha}{\alpha}\right)^{k} \left(\frac{s^\alpha-r^\alpha}{\alpha}\right)^{n-k}. $$
\end{cor}

\begin{proof}
This follows immediately from the theorem if we take $f(t) = \frac{1}{n!}\left(\frac{t^\alpha-r^\alpha}{\alpha}\right)^{n}$ in Taylor's formula. It can also be shown directly.
\end{proof}


\section{steffensen inequality}

We begin this section with a definition of  $\alpha$-fractional integrability. The results in this and subsequent sections differ from those in \cite{pps}.

\begin{definition} 
	\index{$\alpha$-fractional integrable}
Let $\alpha\in(0,1]$ and $0 \le a < b$. A function $f:[a,b]\rightarrow\R$ is $\alpha$-fractional integrable on $[a,b]$ if the integral
$$ \int_a^b f(t) d_{\alpha}t := \int_a^b f(t)t^{\alpha-1}dt $$
exists and is finite.
\end{definition}


\begin{lemma}
Let $\alpha\in(0,1]$ and $a,b\in\R$ with $0 \le a < b$. Let $g:[a,b]\rightarrow[0,1]$ be an $\alpha$-fractional integrable function on $[a,b]$, and define 
\begin{equation}\label{elldef}
 \ell:=\frac{\alpha(b-a)}{b^{\alpha}-a^{\alpha}}\int_{a}^{b} g(t) d_{\alpha}t \in[0, b-a].
\end{equation}
Then
\begin{equation}\label{gbds}
 \int_{b-\ell}^{b} 1\; d_{\alpha}t \le \int_{a}^{b} g(t)\; d_{\alpha}t \le \int_{a}^{a+\ell} 1\; d_{\alpha}t.
\end{equation}
\end{lemma}

\begin{proof}
Since $g(t)\in[0,1]$ for all $t\in[a,b]$, $\ell$ given in \eqref{elldef} satisfies 
\[ 0\le \ell = \frac{\alpha(b-a)}{b^{\alpha}-a^{\alpha}}\int_{a}^{b} g(t) d_{\alpha}t \le \frac{\alpha(b-a)}{b^{\alpha}-a^{\alpha}}\int_{a}^{b} 1\; d_{\alpha}t =  \frac{\alpha(b-a)}{b^{\alpha}-a^{\alpha}}\frac{b^{\alpha}-a^{\alpha}}{\alpha} = b-a. \]
As $\alpha\in(0,1]$ we have that $t^{\alpha-1}$ is a decreasing function on $[a,b]$, or $(a,b]$ if $a=0$. Thus using the fact that $d_{\alpha}t=t^{\alpha-1}dt$ we have the following inequalities, which are average values, namely
\[ \frac{1}{\ell} \int_{b-\ell}^{b} 1\; d_{\alpha}t \le \frac{1}{b-a}\int_a^b 1\; d_{\alpha}t \le \frac{1}{\ell} \int_a^{a+\ell} 1\; d_{\alpha}t. \] 
This implies that
\[  \int_{b-\ell}^{b} 1\; d_{\alpha}t \le \frac{\ell}{b-a}\int_a^b 1\; d_{\alpha}t \le \int_a^{a+\ell} 1\; d_{\alpha}t, \]
which leads to \eqref{gbds} via \eqref{elldef}.
\end{proof}


\begin{theorem}[Fractional Steffensen Inequality]\label{t4}
Let $\alpha\in(0,1]$ and $a,b\in\R$ with $0 \le a < b$. Let $f:[a,b]\rightarrow[0,\infty)$ and $g:[a,b]\rightarrow[0,1]$ be $\alpha$-fractional integrable functions on $[a,b]$, with $f$ decreasing. Then
\begin{equation}\label{steff}
  \int_{b-\ell}^b f(t)d_{\alpha}t \le \int_a^b f(t)g(t)d_{\alpha}t \le \int_a^{a+\ell} f(t)d_{\alpha}t,
\end{equation}
where $\ell$ is given by \eqref{elldef}. 
\end{theorem}

\begin{proof}
We will prove only the case in \eqref{steff} for the left inequality; the proof for the right inequality is similar, and relies on \eqref{gbds}. 

By the definition of $\ell$ in \eqref{elldef} and the conditions on $g$, we know that \eqref{gbds} holds. After subtracting within the left inequality of \eqref{steff}, we see that
\begin{eqnarray*}
 \int_a^b f(t)g(t)d_{\alpha}t 
   &-&   \int_{b-\ell}^b f(t)d_{\alpha}t \\
   &=&   \int_a^{b-\ell} f(t)g(t)d_{\alpha}t + \int_{b-\ell}^b f(t)g(t)d_{\alpha}t - \int_{b-\ell}^b f(t)d_{\alpha}t \\
   &=&   \int_a^{b-\ell} f(t)g(t)d_{\alpha}t - \int_{b-\ell}^b f(t)(1-g(t))d_{\alpha}t \\
   &\ge& \int_a^{b-\ell} f(t)g(t)d_{\alpha}t - f(b-\ell)\int_{b-\ell}^b (1-g(t))d_{\alpha}t \\
   &=&   \int_a^{b-\ell} f(t)g(t)d_{\alpha}t - f(b-\ell)\left[\int_{b-\ell}^b 1\;d_{\alpha}t-\int_{b-\ell}^b g(t)d_{\alpha}t\right] \\ 
   &\stackrel{\eqref{gbds}}{\ge}& \int_a^{b-\ell} f(t)g(t)d_{\alpha}t - f(b-\ell)\int_a^{b-\ell} g(t)d_{\alpha}t  \\
   &=&   \int_a^{b-\ell} \left(f(t)-f(b-\ell)\right)g(t)d_{\alpha}t \ge 0,
\end{eqnarray*}
since $f$ is decreasing and $g$ is nonnegative. Therefore the left-hand side of \eqref{steff} holds.
\end{proof}


\begin{remark}
The requirement in Steffensen's Theorem \ref{t4} that $f$ be non-negative is essential. For example, let $a=0$, $b=1$, $\alpha=1/2\equiv g$, and $f\equiv -1$. Then $\ell=1/2$, and
\[ \int_{b-\ell}^b f(t)d_{\alpha}t =-2+\sqrt{2} > -1 = \int_a^b f(t)g(t)d_{\alpha}t, \]
a contradiction of the left-hand side of \eqref{steff}.
\end{remark}


\section{taylor remainder}

Let $\alpha\in(0,1]$ and suppose $f$ is $n+1$ times $\alpha$-fractional differentiable on $[0,\infty)$. Using Taylor's Theorem, Theorem \ref{t144}, we define the remainder function 
by 
$$ R_{-1,f}(\cdot,s):=f(s), $$ 
and for $n>-1$,
\begin{eqnarray}
 R_{n,f}(t,s) &:=& f(s) - \sum_{k=0}^n \frac{D^{k}_{\alpha}f(t)}{k!}\left(\frac{s^\alpha-t^\alpha}{\alpha}\right)^{k} \label{remain} \\
              &= & \frac{1}{n!} \int_{t}^s \left(\frac{s^\alpha-\tau^\alpha}{\alpha}\right)^n D^{n+1}_{\alpha}f(\tau)d_{\alpha}\tau. \nonumber
\end{eqnarray}


\begin{lemma}\label{l6}
Let $\alpha\in(0,1]$. The following identity involving $\alpha$-fractional Taylor's remainder holds:
$$ \int_a^b \frac{D^{n+1}_{\alpha}f(s)}{(n+1)!}\left(\frac{t^\alpha-s^\alpha}{\alpha}\right)^{n+1} d_{\alpha}s = \int_a^t R_{n,f}(a,s)d_{\alpha}s + \int_t^b R_{n,f}(b,s)d_{\alpha}s. $$
\end{lemma}

\begin{proof}
We proceed by mathematical induction on $n$.  For $n=-1$, 
$$ \int_a^b D^{0}_{\alpha}f(s)d_{\alpha}s = \int_a^b f(s)d_{\alpha}s=\int_a^tf(s)d_{\alpha}s+\int_t^bf(s)d_{\alpha}s. $$
Assume the result holds for $n=k-1$:
$$ \int_a^b \frac{D^{k}_{\alpha}f(s)}{k!}\left(\frac{t^\alpha-s^\alpha}{\alpha}\right)^{k} d_{\alpha}s = \int_a^t R_{k-1,f}(a,s)d_{\alpha}s
              + \int_t^b R_{k-1,f}(b,s)d_{\alpha}s.  $$
Let $n=k$. Using integration by parts, we have
\begin{eqnarray*} 
 \int_a^b \frac{D^{k+1}_{\alpha}f(s)}{(k+1)!}\left(\frac{t^\alpha-s^\alpha}{\alpha}\right)^{k+1} d_{\alpha}s 
 &=& \frac{D^{k}_{\alpha}f(b)}{(k+1)!}\left(\frac{t^\alpha-b^\alpha}{\alpha}\right)^{k+1} \\
 & & \; - \frac{D^{k}_{\alpha}f(a)}{(k+1)!}\left(\frac{t^\alpha-a^\alpha}{\alpha}\right)^{k+1} \\
 & & \; + \int_a^b \frac{D^{k}_{\alpha}f(s)}{k!}\left(\frac{t^\alpha-s^\alpha}{\alpha}\right)^{k} d_{\alpha}s.
\end{eqnarray*}
By the induction assumption, 
\begin{eqnarray*}
 \int_a^b \frac{D^{k+1}_{\alpha}f(s)}{(k+1)!}\left(\frac{t^\alpha-s^\alpha}{\alpha}\right)^{k+1} d_{\alpha}s 
  &=& \int_a^t R_{k-1,f}(a,s)d_{\alpha}s + \int_t^b R_{k-1,f}(b,s)d_{\alpha}s \\
  & & + \frac{D^{k}_{\alpha}f(b)}{(k+1)!}\left(\frac{t^\alpha-b^\alpha}{\alpha}\right)^{k+1} \\
  & & - \frac{D^{k}_{\alpha}f(a)}{(k+1)!}\left(\frac{t^\alpha-a^\alpha}{\alpha}\right)^{k+1} \\
  &=& \int_a^t R_{k-1,f}(a,s)d_{\alpha}s + \int_t^b R_{k-1,f}(b,s)d_{\alpha}s \\
  & & + \frac{D^{k}_{\alpha}f(b)}{k!}\int_b^t \left(\frac{s^\alpha-b^\alpha}{\alpha}\right)^{k} d_{\alpha}s \\
  & & - \frac{D^{k}_{\alpha}f(a)}{k!}\int_a^t \left(\frac{s^\alpha-a^\alpha}{\alpha}\right)^{k} d_{\alpha}s \\
  &=& \int_a^t \left[R_{k-1,f}(a,s) - \frac{D^{k}_{\alpha}f(a)}{k!}\left(\frac{s^\alpha-a^\alpha}{\alpha}\right)^{k} \right]d_{\alpha}s \\
  & & + \int_t^b \left[R_{k-1,f}(b,s) - \frac{D^{k}_{\alpha}f(b)}{k!}\left(\frac{s^\alpha-b^\alpha}{\alpha}\right)^{k} \right] d_{\alpha}s \\ 
  &=& \int_a^t R_{k,f}(a,s)d_{\alpha}s + \int_t^b R_{k,f}(b,s) d_{\alpha}s.
\end{eqnarray*}
This completes the proof.
\end{proof}

\begin{cor}\label{c7}
Let $\alpha\in(0,1]$. For $n\ge -1$,
\begin{eqnarray*}  
  \int_a^b \frac{D^{n+1}_{\alpha}f(s)}{(n+1)!}\left(\frac{a^\alpha-s^\alpha}{\alpha}\right)^{n+1} d_{\alpha}s = \int_a^b R_{n,f}(b,s)d_{\alpha}s, \\
  \int_a^b \frac{D^{n+1}_{\alpha}f(s)}{(n+1)!}\left(\frac{b^\alpha-s^\alpha}{\alpha}\right)^{n+1} d_{\alpha}s = \int_a^b R_{n,f}(a,s)d_{\alpha}s.
\end{eqnarray*}
\end{cor}


\section{applications of the steffensen inequality}

Let $\alpha\in(0,1]$. In the following we adapt to the $\alpha$-fractional setting some results from \cite{gauchman} by applying the fractional Steffensen inequality, Theorem \ref{t4}.


\begin{theorem}\label{t9}
Let $\alpha\in(0,1]$ and $f:[a,b]\rightarrow\R$ be an $n+1$ times $\alpha$-fractional differentiable function such that $D^{n+1}_{\alpha}f$ is increasing and $D^{n}_{\alpha}f$ is decreasing on $[a,b]$. If
\[ \ell:= \frac{b-a}{n+2}, \]
then 
\begin{eqnarray*}
 D^{n}_{\alpha}f(a+\ell) - D^{n}_{\alpha}f(a) 
 &\le& (n+1)!\left(\frac{\alpha}{b^{\alpha}-a^{\alpha}}\right)^{n+1} \int_a^b R_{n,f}(a,s)d_{\alpha}s \\
 &\le& D^{n}_{\alpha}f(b) - D^{n}_{\alpha}f(b-\ell). 
\end{eqnarray*}
\end{theorem}

\begin{proof}
Let $F:=-D^{n+1}_{\alpha}f$. Because $D^{n}_{\alpha}f$ is decreasing, $D^{n+1}_{\alpha}f\le 0$, so that $F\ge 0$ and decreasing on $[a,b]$. Define
\[ g(t):=\left(\frac{b^{\alpha}-t^{\alpha}}{b^{\alpha}-a^{\alpha}}\right)^{n+1} \in[0,1], \quad t\in[a,b], \quad n\ge -1. \]
Note that $F,g$ satisfy the assumptions of Steffensen's inequality, Theorem \ref{t4}; using \eqref{elldef},
\[ \ell = \frac{\alpha(b-a)}{b^{\alpha}-a^{\alpha}}\int_{a}^{b} g(t) d_{\alpha}t=\frac{b-a}{n+2}, \]
and 
$$ -\int_{b-\ell}^b D^{n+1}_{\alpha}f(t) d_{\alpha}t \le -\int_a^b D^{n+1}_{\alpha}f(t)\left(\frac{b^{\alpha}-t^{\alpha}}{b^{\alpha}-a^{\alpha}}\right)^{n+1} d_{\alpha}t \le -\int_a^{a+\ell} D^{n+1}_{\alpha}f(t)d_{\alpha}t. $$
By Corollary \ref{c7} this simplifies to
\[ D^{n}_{\alpha}f(t)|_{t=a}^{a+\ell} \le (n+1)!\left(\frac{\alpha}{b^{\alpha}-a^{\alpha}}\right)^{n+1} \int_a^b R_{n,f}(a,t)d_{\alpha}t
      \le D^{n}_{\alpha}f(t)|^b_{t=b-\ell}. \]
This completes the proof.
\end{proof}

The following corollary is the first Hermite-Hadamard inequality, derived from Theorem \ref{t9} with $n=0$.


\begin{cor}[Hermite-Hadamard Inequality I]\label{c10}
Let $\alpha\in(0,1]$ and $f:[a,b]\rightarrow\R$ be an $\alpha$-fractional differentiable function such that $D_{\alpha}f$ is increasing and $f$ is decreasing on $[a,b]$. Then
\begin{eqnarray*}
 f\left(\frac{a+b}{2}\right) \le \frac{\alpha}{b^{\alpha}-a^{\alpha}} \int_a^b f(s)d_{\alpha}s \le f(b) + f(a) - f\left(\frac{a+b}{2}\right). 
\end{eqnarray*}
\end{cor}


\begin{theorem}\label{t12}
Let $\alpha\in(0,1]$ and $f:[a,b]\rightarrow\R$ be an $n+1$ times $\alpha$-fractional differentiable function such that 
\[ m \le D^{n+1}_{\alpha}f \le M \] 
on $[a,b]$ for some real numbers $m<M$. Then 
\begin{eqnarray}
 \frac{m}{(n+2)!} \left(\frac{b^{\alpha}-a^{\alpha}}{\alpha}\right)^{n+2} + \frac{M-m}{(n+2)!}\left(\frac{b^{\alpha}-(b-\ell)^{\alpha}}{\alpha}\right)^{n+2} \le \int_a^b R_{n,f}(a,t)d_{\alpha}t \nonumber \\
   \le \frac{M}{(n+2)!} \left(\frac{b^{\alpha}-a^{\alpha}}{\alpha}\right)^{n+2} + \frac{m-M}{(n+2)!}\left(\frac{b^{\alpha}-(a+\ell)^{\alpha}}{\alpha}\right)^{n+2}, \label{t12result}
\end{eqnarray}
where $\ell$ is given by
\[ \ell = \frac{\alpha(b-a)}{\left(b^{\alpha}-a^{\alpha}\right)(M-m)} \left(D^{n}_{\alpha}f(b) - D^{n}_{\alpha}f(a) - m\left(\frac{b^{\alpha}-a^{\alpha}}{\alpha}\right)\right) \]
\end{theorem}

\begin{proof}
Let 
\begin{eqnarray*}
 F(t)&:=&\frac{1}{(n+1)!}\left(\frac{b^{\alpha}-t^{\alpha}}{\alpha}\right)^{n+1}, \\
 k(t)&:=&\frac 1{M-m}\left(f(t)-\frac{m}{(n+1)!}\left(\frac{t^{\alpha}-a^{\alpha}}{\alpha}\right)^{n+1}\right),  \\
 G(t)&:=&D^{n+1}_{\alpha}k(t)=\frac 1{M-m}\left(D^{n+1}_{\alpha}f(t)-m\right)\in[0,1].
\end{eqnarray*}
Observe that $F$ is nonnegative and decreasing, and
\[ \int_a^b G(t)d_{\alpha}t = \frac 1{M-m}\left(D^{n}_{\alpha}f(b) - D^{n}_{\alpha}f(a) - m\left(\frac{b^{\alpha}-a^{\alpha}}{\alpha}\right)\right). \]
Since $F,G$ satisfy the hypotheses of Theorem \ref{t4}, we compute the various integrals given in \eqref{steff}, after using \eqref{elldef} to set
\[ \ell=\frac{\alpha(b-a)}{b^{\alpha}-a^{\alpha}}\int_{a}^{b} G(t) d_{\alpha}t. \]
We have
\[ \int_{b-\ell}^b F(t)d_{\alpha}t = \int_{b-\ell}^b \frac{1}{(n+1)!}\left(\frac{b^{\alpha}-t^{\alpha}}{\alpha}\right)^{n+1} d_{\alpha}t 
                              = \frac{1}{(n+2)!}\left(\frac{b^{\alpha}-(b-\ell)^{\alpha}}{\alpha}\right)^{n+2}, \]
and
\[ \int_a^{a+\ell} F(t)d_{\alpha}t = \frac{1}{(n+2)!}\left[\left(\frac{b^{\alpha}-a^{\alpha}}{\alpha}\right)^{n+2}-\left(\frac{b^{\alpha}-(a+\ell)^{\alpha}}{\alpha}\right)^{n+2}\right]. \]
Moreover, using Corollary \ref{c7}, we have
\begin{eqnarray*}
  \int_a^b F(t)G(t) d_{\alpha}t 
    &=& \frac 1{(M-m)(n+1)!}\int_a^b \left(\frac{b^{\alpha}-t^{\alpha}}{\alpha}\right)^{n+1}\left(D^{n+1}_{\alpha}f(t)-m\right)d_{\alpha}t \\
    &=& \frac 1{M-m}\int_a^b R_{n,f}(a,t)d_{\alpha}t - \frac m{(M-m)(n+2)!}\left(\frac{b^{\alpha}-a^{\alpha}}{\alpha}\right)^{n+2}.
\end{eqnarray*}
Using Steffensen's inequality \eqref{steff} and some rearranging, we obtain \eqref{t12result}.
\end{proof}


\begin{cor}\label{t13}
Let $\alpha\in(0,1]$ and $f:[a,b]\rightarrow\R$ be an $\alpha$-fractional differentiable function such that 
\[ m \le D_{\alpha}f \le M \] 
on $[a,b]$ for some real numbers $m<M$. Then
\begin{eqnarray}
 \frac{m}{2} \left(\frac{b^{\alpha}-a^{\alpha}}{\alpha}\right)^{2} + \frac{M-m}{2}\left(\frac{b^{\alpha}-(b-\ell)^{\alpha}}{\alpha}\right)^{2} 
 \le \int_a^b f(t)d_{\alpha}t - f(a)\left(\frac{b^{\alpha}-a^{\alpha}}{\alpha}\right) \nonumber \\
   \le \frac{M}{2} \left(\frac{b^{\alpha}-a^{\alpha}}{\alpha}\right)^{2} + \frac{m-M}{2}\left(\frac{b^{\alpha}-(a+\ell)^{\alpha}}{\alpha}\right)^{2}, \label{t13result}
\end{eqnarray}
where $\ell$ is given by
\[ \ell = \frac{\alpha(b-a)}{\left(b^{\alpha}-a^{\alpha}\right)(M-m)} \left(f(b) - f(a) - m\left(\frac{b^{\alpha}-a^{\alpha}}{\alpha}\right)\right) \]
\end{cor}

\begin{proof}
Use the previous theorem with $n=0$, and Corollary \ref{c7}.
\end{proof}


\section{applications of the $\breve{\text{C}}$eby$\breve{\text{s}}$ev inequality}

Let $\alpha\in(0,1]$. We begin with $\breve{\text{C}}$eby$\breve{\text{s}}$ev's inequality for $\alpha$-fractional integrals, then apply it to obtain a Hermite-Hadamard-type inequality.


\begin{theorem}[$\breve{\text{C}}$eby$\breve{\text{s}}$ev Inequality]\label{t16}
Let $f$ and $g$ be both increasing or both decreasing in $[a,b]$, and let $\alpha\in(0,1]$.  Then
\[ \int_a^b f(t)g(t)d_{\alpha}t \ge \frac {\alpha}{b^\alpha-a^\alpha}\int_a^bf(t)d_{\alpha}t\int_a^b g(t)d_{\alpha}t. \]
If one of the functions is increasing and the other is decreasing, then the above inequality is reversed.
\end{theorem}

\begin{proof}
The proof is very similar to the classical case with $\alpha=1$.
\end{proof}

The following is an application of $\breve{\text{C}}$eby$\breve{\text{s}}$ev's inequality, which extends a similar result in \cite{gauchman} for $q$-calculus to this $\alpha$-fractional case.


\begin{theorem}\label{t17}
Let $\alpha\in(0,1]$. Assume that $D^{n+1}_{\alpha}f$ is monotonic on $[a,b]$. If $D^{n+1}_{\alpha}f$ is increasing, then
\begin{eqnarray*}
 0 &\ge& \int_a^b R_{n,f}(a,t) d_{\alpha}t 
    - \left(\frac{D^{n}_{\alpha}f(b)-D^{n}_{\alpha}f(a)}{(n+2)!}\right) \left(\frac{b^\alpha-a^\alpha}{\alpha}\right)^{n+1} \\
   &\ge& \left(\frac{D^{n+1}_{\alpha}f(a)-D^{n+1}_{\alpha}f(b)}{(n+2)!}\right) \left(\frac{b^\alpha-a^\alpha}{\alpha}\right)^{n+2}.
\end{eqnarray*}
If $D^{n+1}_{\alpha}f$ is decreasing, then the inequalities are reversed.
\end{theorem}

\begin{proof}
The situation where $D^{n+1}_{\alpha}f$ is decreasing is analogous to that of $D^{n+1}_{\alpha}f$ increasing. Thus, assume $D^{n+1}_{\alpha}f$ is increasing and set 
\[ F(t):=D^{n+1}_{\alpha}f(t), \quad G(t):=\frac{1}{(n+1)!} \left(\frac{b^\alpha-t^\alpha}{\alpha}\right)^{n+1}. \] 
Then $F$ is increasing by assumption, and $G$ is decreasing, so that by $\breve{\text{C}}$eby$\breve{\text{s}}$ev's inequality,
$$ \int_a^b F(t)G(t)d_{\alpha}t \le \frac {\alpha}{b^\alpha-a^\alpha}\int_a^b F(t)d_{\alpha}t \int_a^b G(t)d_{\alpha}t. $$
By Corollary \ref{c7}, 
\[  \int_a^bF(t)G(t)d_{\alpha}t = \int_a^b \frac{D^{n+1}_{\alpha}f(t)}{(n+1)!}\left(\frac{b^\alpha-t^\alpha}{\alpha}\right)^{n+1} d_{\alpha}t = \int_a^b R_{n,f}(a,t)d_{\alpha}t. \]
We also have
\[ \int_a^b F(t)d_{\alpha}t = D^{n}_{\alpha}f(b)-D^{n}_{\alpha}f(a), \quad
 \int_a^b G(t)d_{\alpha}t = \frac{1}{(n+2)!} \left(\frac{b^\alpha-a^\alpha}{\alpha}\right)^{n+2}. \]
Thus $\breve{\text{C}}$eby$\breve{\text{s}}$ev's inequality implies
\[ \int_a^b R_{n,f}(a,t) d_{\alpha}t 
     \le \frac {\alpha}{b^\alpha-a^\alpha}\left(D^{n}_{\alpha}f(b)-D^{n}_{\alpha}f(a)\right)\frac{1}{(n+2)!} \left(\frac{b^\alpha-a^\alpha}{\alpha}\right)^{n+2}, \]
which subtracts to the left side of the inequality. Since $D^{n+1}_{\alpha}f$ is increasing on $[a,b]$,
\begin{eqnarray*}
 \frac{D^{n+1}_{\alpha}f(a)}{(n+2)!} \left(\frac{b^\alpha-a^\alpha}{\alpha}\right)^{n+2} 
 &\le& \left(\frac{D^{n}_{\alpha}f(b)-D^{n}_{\alpha}f(a)}{(n+2)!}\right) \left(\frac{b^\alpha-a^\alpha}{\alpha}\right)^{n+1} \\
 &\le& \frac{D^{n+1}_{\alpha}f(b)}{(n+2)!} \left(\frac{b^\alpha-a^\alpha}{\alpha}\right)^{n+2}, 
\end{eqnarray*}
and we have
\begin{eqnarray*} 
 \int_a^b R_{n,f}(a,t)d_{\alpha}t - \left(\frac{D^{n}_{\alpha}f(b)-D^{n}_{\alpha}f(a)}{(n+2)!}\right) \left(\frac{b^\alpha-a^\alpha}{\alpha}\right)^{n+1} \\
   \ge \int_a^b R_{n,f}(a,t) d_{\alpha}t - \frac{D^{n+1}_{\alpha}f(b)}{(n+2)!} \left(\frac{b^\alpha-a^\alpha}{\alpha}\right)^{n+2}. 
\end{eqnarray*}
Now Corollary \ref{c7} and $D^{n+1}_{\alpha}f$ is increasing imply that
\begin{eqnarray*}
 \int_a^b \frac{D^{n+1}_{\alpha}f(b)}{(n+1)!} \left(\frac{b^\alpha-t^\alpha}{\alpha}\right)^{n+1}d_{\alpha}t 
 &\ge& \int_a^b R_{n,f}(a,t)d_{\alpha}t \\
 &\ge& \int_a^b \frac{D^{n+1}_{\alpha}f(a)}{(n+1)!} \left(\frac{b^\alpha-t^\alpha}{\alpha}\right)^{n+1}d_{\alpha}t, 
\end{eqnarray*}
which simplifies to
$$ \frac{D^{n+1}_{\alpha}f(b)}{(n+2)!} \left(\frac{b^\alpha-a^\alpha}{\alpha}\right)^{n+2} \ge \int_a^b R_{n,f}(a,t)d_{\alpha}t
    \ge \frac{D^{n+1}_{\alpha}f(a)}{(n+2)!} \left(\frac{b^\alpha-a^\alpha}{\alpha}\right)^{n+2}. $$
This, together with the earlier lines give the right side of the inequality.
\end{proof}

Compare the following corollary with Corollary \ref{c10}.


\begin{cor}[Hermite-Hadamard Inequality II]\label{hhII}
Let $\alpha\in(0,1]$. If $D_{\alpha}f$ is increasing on $[a,b]$, then
\begin{equation}\label{hermite3}
 \frac{\alpha}{b^\alpha-a^\alpha} \int_a^b f(t) d_{\alpha}t \le \frac{f(b)+f(a)}{2}. 
\end{equation}
If $D_{\alpha}f$ is decreasing on $[a,b]$, then the inequalities are reversed.
\end{cor}


\section{Ostrowski Inequality}

In this section we prove Ostrowski's $\alpha$-fractional inequality using a Montgomery identity. For more on Ostrowski's inequalities, see \cite{bm1} and the references therein.


\begin{lemma}[Montgomery Identity]\label{monty}
Let $a,b,s,t\in\R$ with $0\le a<b$, and let $f:[a,b]\rightarrow\R$ be $\alpha$-fractional differentiable for $\alpha\in(0,1]$. Then
\begin{equation}\label{61}
 f(t) = \frac{\alpha}{b^\alpha-a^\alpha} \int_a^b f(s)d_{\alpha}s + \frac{\alpha}{b^\alpha-a^\alpha} \int_a^b p(t,s)D_{\alpha}f(s)d_{\alpha}s 
\end{equation}
where
\begin{equation}\label{62}
 p(t,s):=\begin{cases} \frac{s^\alpha-a^\alpha}{\alpha} &:a\le s < t, \\ \frac{s^\alpha-b^\alpha}{\alpha} &: t\le s\le b. \end{cases}
\end{equation}
\end{lemma}

\begin{proof}
Integrating by parts, we have
\[ \int_a^t \left(\frac{s^\alpha-a^\alpha}{\alpha}\right)D_{\alpha} f(s)d_{\alpha}s = \frac{t^\alpha-a^\alpha}{\alpha}f(t) - \int_a^t f(s)d_{\alpha}s \]
and
\[ \int_t^b \left(\frac{s^\alpha-b^\alpha}{\alpha}\right)D_{\alpha} f(s)d_{\alpha}s = \frac{b^\alpha-t^\alpha}{\alpha}f(t) - \int_t^b f(s)d_{\alpha}s. \]
Adding and solving for $f$ yields the result.
\end{proof}


\begin{theorem}[Ostrowski Inequality]
Let $a,b,s,t\in\R$ with $0\le a<b$, and let $f:[a,b]\rightarrow\R$ be $\alpha$-fractional differentiable for $\alpha\in(0,1]$. Then
\begin{equation}\label{63}
 \left|f(t) - \frac{\alpha}{b^\alpha-a^\alpha} \int_a^b f(t)d_{\alpha}t \right| \le \frac{M}{2\alpha\left(b^\alpha-a^\alpha\right)} \left[\left(t^\alpha-a^\alpha\right)^2 + \left(b^\alpha-t^\alpha\right)^2 \right],
\end{equation}
where
\[ M:=\sup_{t\in(a,b)} \left|D_{\alpha}f(t)\right|. \]
This inequality is sharp in the sense that the right-hand side of \eqref{63} cannot be replaced by a smaller one.
\end{theorem}

\begin{proof}
Using Lemma \ref{monty} with $p(t,s)$ defined in \ref{62} we see that
\begin{eqnarray*}
 \left|f(t) - \frac{\alpha}{b^\alpha-a^\alpha} \int_a^b f(s)d_{\alpha}s\right|
  & = & \left|\frac{\alpha}{b^\alpha-a^\alpha} \int_a^b p(t,s)D_{\alpha}f(s)d_{\alpha}s\right| \\
  &\le& \frac{M\alpha}{b^\alpha-a^\alpha} \left(\int_a^t \left|\frac{s^\alpha-a^\alpha}{\alpha}\right| d_{\alpha}s + \int_t^b \left|\frac{s^\alpha-b^\alpha}{\alpha}\right| d_{\alpha}s\right) \\ 
  & = & \frac{M\alpha}{b^\alpha-a^\alpha} \left(\int_a^t \left(\frac{s^\alpha-a^\alpha}{\alpha}\right) d_{\alpha}s + \int_t^b \left(\frac{b^\alpha-s^\alpha}{\alpha}\right) d_{\alpha}s\right) \\
  & = & \frac{M\alpha}{b^\alpha-a^\alpha} \left(\frac 12\left(\frac{s^\alpha-a^\alpha}{\alpha}\right)^2\Big|_{a}^{t} -\frac 12 \left(\frac{b^\alpha-s^\alpha}{\alpha}\right)^2\Big|_{t}^{b} \right) \\
  & = & \frac{M}{2\alpha\left(b^\alpha-a^\alpha\right)} \left[\left(t^\alpha-a^\alpha\right)^2 + \left(b^\alpha-t^\alpha\right)^2 \right].
\end{eqnarray*}
Now $p(t,a) = 0$, so the smallest value attaining the supremum in $M$ is greater than $a$. To prove the sharpness of this inequality, let $f(t) = t^\alpha/\alpha$, $a = t_1$, $b=t_2=t$. It follows that $D_{\alpha}f(t) = 1$ and $M = 1$. Examining the right-hand side of \eqref{63} we get
\[ \frac{M}{2\alpha\left(b^\alpha-a^\alpha\right)} \left[\left(t^\alpha-a^\alpha\right)^2 + \left(b^\alpha-t^\alpha\right)^2 \right] 
 = \frac{\left(t_2^\alpha-t_1^\alpha\right)^2}{2\alpha\left(t_2^\alpha-t_1^\alpha\right)} = \frac{t_2^\alpha-t_1^\alpha}{2\alpha}. \]
Starting with the left-hand side of \eqref{63}, we have
\begin{eqnarray*}
 \left|f(t) - \frac{\alpha}{b^\alpha-a^\alpha} \int_a^b f(t)d_{\alpha}t \right|
 & = & \left|\frac{t^\alpha}{\alpha} - \frac{\alpha}{t_2^\alpha-t_1^\alpha} \int_{t_1}^{t_2}\frac{t^\alpha}{\alpha}d_{\alpha}t \right| \\
 & = & \left|\frac{t^\alpha}{\alpha} - \left(\frac{\alpha}{t_2^\alpha-t_1^\alpha}\right)\left(\frac{t^{2\alpha}}{2\alpha^2}\right)\Big|_{t_1}^{t_2} \right| \\ 
 & = & \left|\frac{t^\alpha}{\alpha} - \left(\frac{1}{t_2^\alpha-t_1^\alpha}\right)\left(\frac{t_2^{2\alpha}-t_1^{2\alpha}}{2\alpha}\right) \right| \\ 
 & = & \left|\frac{t^\alpha}{\alpha} - \left(\frac{t_2^{\alpha}+t_1^{\alpha}}{2\alpha}\right) \right| \\ 
 & = & \frac{t_2^\alpha-t_1^\alpha}{2\alpha}.
\end{eqnarray*}
Therefore by the squeeze theorem the sharpness of Ostrowski's inequality is shown.
\end{proof}


\section{Gr\"{u}ss Inequality}

In this section we prove the Gr\"{u}ss inequality, which relies on Jensen's inequality. Our approach is similar to that taken by \cite{bm2}.


\begin{theorem}[Jensen Inequality]\label{thm71}
Let $\alpha\in(0,1]$ and $a,b,x,y\in[0,\infty)$. If $w:\R\rightarrow\R$ and $g:\R\rightarrow(x,y)$ are nonnegative, continuous functions with $\int_a^b w(t) d_{\alpha}t>0$, and $F:(x,y)\rightarrow\R$ is continuous and convex, then
\[ F\left(\frac{\int_a^b w(t)g(t)d_{\alpha}t}{\int_a^b w(t)d_{\alpha}t}\right) \le \frac{\int_a^b w(t)F(g(t))d_{\alpha}t}{\int_a^b w(t)d_{\alpha}t}. \]
\end{theorem}

\begin{proof}
The proof is the same as those found in Bohner and Peterson \cite[Theorem 6.17]{bp1} and Rudin \cite[Theorem 3.3]{rudin} and thus is omitted.
\end{proof}


\begin{theorem}[Gr\"{u}ss Inequality]\label{gruss}
Let $a,b,s\in[0,\infty)$, and let $f,g:[a,b]\rightarrow\R$ be continuous functions. Then for $\alpha\in(0,1]$ and
\begin{equation}\label{72}
 m_1 \le f(t) \le M_1, \qquad m_2 \le g(t) \le M_2, 
\end{equation}
we have
\begin{eqnarray*}
 \left|\frac{\alpha}{b^\alpha-a^\alpha}\int_a^b f(t)g(t)d_{\alpha}t - \left(\frac{\alpha}{b^\alpha-a^\alpha}\right)^2 \int_a^b f(t)d_{\alpha}t\int_a^b g(t)d_{\alpha}t\right| \\
 \le \frac 14(M_1-m_1)(M_2-m_2). 
\end{eqnarray*}
\end{theorem}

\begin{proof}
Initially we consider an easier case, namely where $f=g$ and
\[ \frac{\alpha}{b^\alpha-a^\alpha}\int_a^b f(t)d_{\alpha}t = 0. \]
If we define
\[ v(t):=\frac{f(t)-m_1}{M_1-m_1}\in[0,1], \]
then $f(t)=m_1+(M_1-m_1)v(t)$. Since
\[ \int_a^b v^2(t)d_{\alpha}t \le \int_a^b v(t) d_{\alpha}t = \frac{-m_1(b^\alpha-a^\alpha)}{\alpha(M_1-m_1)}, \]
we have
\begin{eqnarray*}
 I(f,f) &:= & \frac{\alpha}{b^\alpha-a^\alpha}\int_a^b f^2(t) d_{\alpha}t - \left(\frac{\alpha}{b^\alpha-a^\alpha}\int_a^b f(t) d_{\alpha}t\right)^2 \\
        & = & \frac{\alpha}{b^\alpha-a^\alpha}\int_a^b \left[m_1+(M_1-m_1)v(t)\right]^2(t) d_{\alpha}t \\
        &\le& -m_1M_1 = \frac 14\left[(M_1-m_1)^2-(M_1+m_1)^2\right] \\
        &\le& \frac 14(M_1-m_1)^2.
\end{eqnarray*}
Now consider the case
\[ r:=\frac{\alpha}{b^\alpha-a^\alpha}\int_a^b f(t)d_{\alpha}t \not= 0, \]
where $r\in\R$. If we take $h(t):=f(t)-r$, then $h(t)\in[m_1-r,M_1-r]$ and
\[ \frac{\alpha}{b^\alpha-a^\alpha}\int_a^b h(t)d_{\alpha}t = \frac{\alpha}{b^\alpha-a^\alpha}\int_a^b (f(t)-r)d_{\alpha}t = r - \frac{r\alpha}{b^\alpha-a^\alpha}\int_a^b d_{\alpha}t = 0. \]
Consequently $h$ satisfies the earlier assumptions and so
\[ I(h,h) \le \frac 14\left[M_1-r-(m_1-r)\right]^2 = \frac 14\left(M_1-m_1\right)^2. \]
Additionally we have
\[ I(h,h) = \frac{\alpha}{b^\alpha-a^\alpha}\int_a^b (f(t)-r)^2 d_{\alpha}t = -r^2 + \frac{\alpha}{b^\alpha-a^\alpha}\int_a^b f^2(t) d_{\alpha}t = I(f,f). \]
As a result,
\[ I(f,f)= I(h,h) \le \frac 14\left(M_1-m_1\right)^2. \]
Let us now turn to the case involving general functions $f$ and $g$ under assumptions \eqref{72}. Using
\begin{eqnarray*}
  I(f,g) &:=& \frac{\alpha}{b^\alpha-a^\alpha}\int_a^b f(t)g(t) d_{\alpha}t - \left(\frac{\alpha}{b^\alpha-a^\alpha}\right)^2 \int_a^b f(t) d_{\alpha}t \int_a^b g(t) d_{\alpha}t 
\end{eqnarray*}
and the earlier cases, one can easily finish the proof as in the case with $\alpha=1$. See \cite{bm2} for complete details to mimic.
\end{proof}


\begin{cor}
Let $\alpha\in(0,1]$, $a,b,s,t\in[0,\infty)$ and $f:[a,b]\rightarrow\R$ be $\alpha$-fractional differentiable. If $D_{\alpha}f$ is continuous and
\[ m \le D_{\alpha}f(t) \le M, \quad t\in[a,b], \]
then
\begin{eqnarray}
 \left|f(t) - \frac{\alpha}{b^\alpha-a^\alpha} \int_a^b f(s)d_{\alpha}s - \left[\frac{2t^\alpha-a^\alpha-b^\alpha}{2(b^\alpha-a^\alpha)}\right] \left[f(b)-f(a)\right]\right| \nonumber \\
 \le \frac 14\left(\frac{b^\alpha-a^\alpha}{\alpha}\right)(M-m). \label{73}
\end{eqnarray}
for all $t\in[a,b]$.
\end{cor}

\begin{proof}
Using Lemma \ref{monty} Montgomery's identity we have
\begin{equation}\label{fpmont}
 f(t) - \frac{\alpha}{b^\alpha-a^\alpha} \int_a^b f(s)d_{\alpha}s = \frac{\alpha}{b^\alpha-a^\alpha} \int_a^b p(t,s)D_{\alpha}f(s)d_{\alpha}s 
\end{equation}
for all $t\in[a,b]$, where $p(t,s)$ is given in \eqref{62}. Now for all $t,s\in[a,b]$, we see that
\[ \frac{t^\alpha-b^\alpha}{\alpha} \le p(t,s) \le \frac{t^\alpha-a^\alpha}{\alpha}. \]
Applying Theorem \ref{gruss} Gr\"{u}ss' inequality to the mappings $p(t,\cdot)$ and $D_{\alpha}f$, we obtain
\begin{eqnarray}
 \left|\frac{\alpha}{b^\alpha-a^\alpha}\int_a^b p(t,s)D_{\alpha}f(s)d_{\alpha}s - \left(\frac{\alpha}{b^\alpha-a^\alpha}\right)^2 \int_a^b p(t,s)d_{\alpha}s \int_a^b D_{\alpha}f(s)d_{\alpha}s\right| \nonumber \\
 \le \frac 14\left(\frac{t^\alpha-a^\alpha}{\alpha}-\frac{t^\alpha-b^\alpha}{\alpha} \right)(M-m) = \frac 14\left(\frac{b^\alpha-a^\alpha}{\alpha}\right)(M-m). \label{74}
\end{eqnarray}
Computing the integrals involved, we obtain
\begin{eqnarray*}
 \left(\frac{\alpha}{b^\alpha-a^\alpha}\right)^2 \int_a^b p(t,s)d_{\alpha}s = \frac{2t^\alpha-a^\alpha-b^\alpha}{2(b^\alpha-a^\alpha)}
\end{eqnarray*}
and
\begin{eqnarray*}
 \int_a^b D_{\alpha}f(s)d_{\alpha}s = f(b)-f(a),
\end{eqnarray*}
so that \eqref{73} holds, after using \eqref{fpmont} and \eqref{74}.
\end{proof}

Compare the following corollary with Corollary \ref{c10} and Corollary \ref{hhII}.


\begin{cor}[Hermite-Hadamard III]
Let $\alpha\in(0,1]$, $a,b,s,t\in[0,\infty)$ and $f:[a,b]\rightarrow\R$ be $\alpha$-fractional differentiable. If $D_{\alpha}f$ is continuous and
\[ m \le D_{\alpha}f(t) \le M, \quad t\in[a,b], \]
then
\[ \left|\frac{f(b)+f(a)}{2} - \frac{\alpha}{b^\alpha-a^\alpha} \int_a^b f(s)d_{\alpha}s \right| \le \frac 14\left(\frac{b^\alpha-a^\alpha}{\alpha}\right)(M-m). \]
for all $t\in[a,b]$.
\end{cor}

\begin{proof}
Take $t=b$ in the previous corollary.
\end{proof}

\end{document}